\documentclass{amsart}

\usepackage{amsthm,amssymb,latexsym,amsmath}
\usepackage[all]{xy}

\newtheorem*{mthm}{Main Theorem}
\newtheorem{theorem}{Theorem} 
\newtheorem{lemma}[theorem]{Lemma}
\newtheorem{proposition}[theorem]{Proposition}
\newtheorem{corollary}[theorem]{Corollary}

\newtheorem{remark}[theorem]{Remark}

\newcommand{\p}[1]{{\mathbb{P}^{#1}}}

\newcommand{\ox}{{\mathcal O}_{X}}
\newcommand{\ol}{{\mathcal O}_{\ell}}
\newcommand{\inext}{{\mathcal E}{\it xt}}
\newcommand{\cali}{{\mathcal I}}
\newcommand{\op}[1]{{\mathcal O}_{\mathbb{P}^{#1}}}
\DeclareMathOperator{\ext}{Ext}
\newcommand{\sing}{\operatorname{Sing}}
\newcommand{\supp}{\operatorname{Supp}}
\newcommand{\coker}{\operatorname{coker}}
\newcommand{\calh}{{\mathcal H}}

\newcommand{\ZZ}{{\mathbb Z}}
\newcommand{\PP}{{\mathbb P}}

\newcommand{\OO}{{\mathcal O}}

\newcommand{\Hom}{\operatorname{Hom}}

\newcommand{\im}{\operatorname{Im}}
\newcommand{\codim}{\operatorname{codim}}

\begin{document}

\title{Singular locus of instanton sheaves on $\PP^3$}

\author{Michael Gargate}
\address{UTFPR, Campus Pato Branco \\
Rua Via do Conhecimento, km 01 \\
85503-390 Pato Branco, PR, Brazil}
\email{michaelgargate@utfpr.edu.br}

\author{Marcos Jardim}
\address{IMECC - UNICAMP \\
Departamento de Matem\'atica \\
Rua S\'ergio Buarque de Holanda, 651 \\
13083-970 Campinas, SP, Brazil}
\email{jardim@ime.unicamp.br}

\begin{abstract}
We prove that the singular locus of rank $2$ instanton sheaf $E$ on $\p3$ which is not locally free has pure dimension $1$. Moreover, we also show that the dual and double dual of $E$ are isomorphic locally free instanton sheaves, and that the sheaves $\inext^1(E,\op3)$ and $E^{\vee\vee}/E$ are rank $0$ instantons. We also provide explicit examples of instanton sheaves of rank $3$ and $4$ illustrating that all of these claims are false for higher rank instanton sheaves. 
\end{abstract}

\maketitle
\tableofcontents


\section{Introduction}

Mathematical instanton bundles have been intensely studied by several authors since its introduction in the late 1970's. They derive their nomenclature from Gauge Theory: mathematical instanton bundles on odd dimensional complex projective spaces $\mathbb{C}\mathbb{P}^{2k+1}$ are precisely those holomorphic vector bundles that arise, via the twistor correspondence, in relation with quaternionic instantons on quaternionic projective spaces $\mathbb{H}\mathbb{P}^{k}$, see \cite{MCS} for details, and also \cite{ok2}.

The simplest case of such objects are rank $2$ instanton bundles on $\mathbb{C}\mathbb{P}^{3}$, and there is a vast literature about them. One outstanding problem that resisted solution for a couple of decades regards the irreducibility and smoothness of the moduli space $\mathcal{I}(c)$ of rank $2$ instanton bundles on $\mathbb{C}\mathbb{P}^{3}$ of charge $c$. It was known since 2003 that  $\mathcal{I}(c)$ is smooth and irreducible for $c\le5$, see \cite{ctt} and the references therein. Recently, Tikhomirov has proved in \cite{T1,T2} irreducibility for arbitrary $c$; while the second named author and Verbitsky established smoothness for every $c$, see \cite{JV}. 

The next step is to understand how instanton bundles degenerate, that is to study non locally free instanton sheaves. Maruyama and Trautmann were the first to consider \emph{limits of instantons} on $\p3$ in \cite{MaTr}; in \cite{J-i} instanton sheaves on arbitary projective spaces $\mathbb{P}^{n}$ are studied. In this paper we only consider instanton sheaves on $\p3$; these are defined as torsion free sheaves $E$ on $\p3$ with $c_1(E)=0$, $c_3(E)=0$, and satisfying
$$ h^0(E(-1))=h^1(E(-2))=h^2(E(-2))=h^3(E(-3))=0 ,$$
compare with \cite[1.1, page 216]{MaTr} and \cite[page 69]{J-i}. The charge of $E$ is given by its second Chern class $c_2(E)$.

The goal of this paper is to study the singular locus of rank $2$ instanton sheaves, showing that they are of pure dimension $1$. In the process, the rank $0$ instantons introduced by Hauzer and Langer in \cite[Definition 6.1]{hauzer} also appear. A rank $0$ instanton is a pure dimension $1$ sheaf $Z$ on $\p3$ such that $h^0(Z(-2))=h^1(Z(-2))=0$; $d:=h^0(Z(-1))$ is called the degree of $Z$; see Section \ref{rk0instantons} below for details. We say that two rank $0$ instantons $Z_1$ and $Z_2$ are dual to each other if $\inext^2(Z_1,\op3)\simeq Z_2$ and 
$\inext^2(Z_2,\op3)\simeq Z_1$.

More precisely, with the previous definitions in mind, we prove the following.

\begin{mthm}\label{mthm}
If $E$ is a rank $2$ non locally free instanton sheaf on $\p3$, then
\begin{enumerate}
\item[(i)] its singular locus has pure dimension $1$;
\item[(ii)] $E^\vee$ and $E^{\vee\vee}$ are isomorphic locally free instanton sheaves;
\item[(iii)] the sheaves $\inext^1(E,\op3)$ and $E^{\vee\vee}/E$ are dual rank $0$ instantons of degree $c_2(E)-c_2(E^{\vee\vee})$.
\end{enumerate}
In addition, if $c_2(E)-c_2(E^{\vee\vee})=1$, then the singular locus of $E$ consists of a single line.
\end{mthm}

All of these claims are false for instanton sheaves of higher rank. Indeed, we provide explicit examples of instanton sheaves of rank $3$ and $4$ on $\p3$ for which $E^\vee$ and $E^{\vee\vee}$ are not locally free and not instanton sheaves, respectively, see Section \ref{sec-examples}. In addition, we also decribe rank $3$ instanton sheaves whose singular loci are either of dimension $0$ or not of pure dimension $1$, see Section \ref{r3c2} below. 

\noindent {\bf Acknowledgements.}
MG was supported by a PhD grant from CAPES, Brazil; most of the results present here were obtained as part of his thesis. MJ is partially supported by the CNPq grant number 302477/2010-1 and the FAPESP grant number 2011/01071-3.


\section{Sheaves and monads}

In this section, we fix the notation that will be used in this work and we recall some basic definitions and results. We work over a fixed algebraically closed field $\kappa$ of characteristic zero. By a projective variety $X$ we understand a nonsingular, projective, integral, separated noetherian scheme of finite type over $\kappa$. All sheaves on $X$ are coherent sheaves of $\ox$-modules.


\subsection{Sheaves}

Let $E$ be a coherent sheaf on projective variety $X$; its \emph{support} is the closed set
\begin{equation}\label{defn-supp}
\supp(E) = \{x\in X ~~|~~ E_x\neq 0\},
\end{equation}
where $E_x$ denotes the stalk of $E$ over the point $x\in X$. The dimension of $E$, denoted by $\dim (E)$, is defined to be the dimension $\supp(E)$ as an algebraic set. A coherent sheaf $E$ is said to be of pure dimension $d$ if $\dim E=d$ and every nonzero subsheaf of $E$ also has dimension $d$. 

The \emph{singular locus} of $E$ is defined as the following closed set
\begin{eqnarray}\label{defn-sing}
\sing(E) & = & \{x\in X ~~|~~ E_x {\rm not~~free}~~ \OO_{X,x}-{\rm module} \} \\
& = & \bigcup_{p=1}^{\dim X} \supp \inext^p(E,\ox). \notag
\end{eqnarray}

Recall also that $E$ is \emph{torsion free} if the canonical map into the double dual sheaf $E^{\vee\vee}$ is injective; if such map is an isomorphism, then $E$ is \emph{reflexive}. One can show that if $E$ is torsion free, then $\codim\sing(E)\ge2$, and if $E$ is reflexive, then $\codim\sing(E)\ge3$.

Two sheaves derived from $E$ will play important parts later on,
\begin{equation}\label{seqe}
S_E:=\inext^1(E,\ox) ~~{\rm and}~~ Q_E:=E^{\vee\vee}/E .
\end{equation}
Clearly, $\supp(S_E)\subseteq\sing(E)$; note also that $\supp(Q_E)\subseteq\sing(E)$. Indeed, if $x\notin\sing(E)$, then $E_x$ is free as an $\OO_{X,x}$-module, hence $E_x\simeq (E^{\vee\vee})_x$ and $x\notin\supp(Q_E)$. Moreover, it is not difficult to see that if $E^{\vee\vee}$ is locally free, then in fact $\supp(Q_E)=\sing(E)$.

Hartshorne proves in \cite{hrt8} the following results that will play key roles in the proof of our main results.

\begin{proposition}{\rm (\cite[Prop. 1.10]{hrt8})} \label{reflexive} 
If $F$ is a rank $2$ reflexive sheaf on $\p3$, then $F^{\vee}\cong F\otimes(\det F)^{-1}$.
\end{proposition}

\begin{proposition}{\rm (\cite[Prop. 2.5]{hrt8})} \label{reflexive1}
If $F$ is a reflexive sheaf on $\mathbb{P}^3$, then there are isomorphisms
$$ H^0(F^\vee(-4)) \simeq H^3(F)^\vee ~~,~~ H^3(F^\vee(-4)) \simeq H^0(F)^\vee $$
and an exact sequence
$$ 0 \to H^1(F^\vee(-4))\to H^2(F)^\vee \to H^0(S_F(-4)) \to H^2(F^\vee(-4)) \to
H^1(F)^\vee \to 0 . $$
\end{proposition}


Another key ingredient is the following criterion due to Roggero.

\begin{proposition}{\rm (cf. \cite[Thm. 2.3]{rog})} \label{prop-rog}
If $F$ is a rank $2$ reflexive sheaf with $c_1(F)=0$ on $\p3$, then $F$ is locally free if and only if $h^2(F(p))=0$ for some $p\leq -2$.
\end{proposition}


\subsection{Monads}

Recall that a \emph{monad} on $X$ is a complex $M^{\bullet}$ of locally free sheaves on $X$ of the following form:
\begin{equation}\label{monad}
M_{\bullet} ~~ : ~~  A \stackrel{\alpha}{\rightarrow} B \stackrel{\beta}{\rightarrow} C
\end{equation}
which is exact on the first and last terms. The sheaf $E=\ker\beta/\im\alpha$ is called the cohomology of the monad $M^{\bullet}$.

The \emph{degeneration locus} of the monad (\ref{monad}) consists of the following set
\begin{equation}\label{deg-locus}
\Delta(M_{\bullet})=\{x\in X ~~ | ~~ \alpha(x) ~~ {\rm is~not~injective} \}.
\end{equation}


\begin{lemma}\label{delta=sing}
If $E$ is the cohomology sheaf of a monad $M^{\bullet}$ as in (\ref{monad}), then
$\Delta(M_{\bullet})=\supp(S_E)=\sing(E)$. In particular, $\supp(Q_E)\subseteq\supp(S_E)$.
\end{lemma}

\begin{proof}
One can break the monad (\ref{monad}) into the short exact sequences
\begin{equation}\label{one}
0 \to K \to B \stackrel{\beta}{\rightarrow} C \to 0
\end{equation}
where $K:=\ker\beta$, and 
\begin{equation}\label{two}
0 \to A \stackrel{\alpha}{\rightarrow} K \to E \to 0
\end{equation}
Dualizing (\ref{two}) we obtain
\begin{equation}\label{dual-two}
0 \to E^\vee \to K^\vee \stackrel{\alpha^\vee}{\rightarrow} A^\vee \to \inext^1(E,\op3) \to 0
\end{equation}
and $\inext^p(E,\op3)=0$ for $p\ge2$. It then follows immediately that $\supp(S_E)=\sing(E)$. To see that $\Delta(M_{\bullet})=\supp(S_E)$, note that $x\in\supp(S_E)$ if and only if the map of stalks $(\alpha^\vee)_x$ is not surjective; this happens if and only if the map of fibers $\alpha^\vee(x)$ is not surjective, which is equivalent to $x\in\Delta(M_{\bullet})$.
\end{proof}

In general, the dual of a monad may not be a monad, only a complex of the form
\begin{equation}\label{dual-monad}
M_{\bullet}^{\vee} ~~ : ~~  C^{\vee} \stackrel{\beta^\vee}{\rightarrow} B^{\vee} \stackrel{\alpha^\vee}{\rightarrow} A^{\vee}
\end{equation}
whose first map is injective. It may have two nontrivial cohomology sheaves:
$\calh^0(M_{\bullet}^{\vee}):=\ker(\alpha^\vee)/\im(\beta^\vee)$ and $\calh^1(M_{\bullet}^{\vee}):=\coker(\alpha^\vee)$.

\begin{lemma}\label{dual-lemma1}
If $E$ is the cohomology sheaf of a monad $M^{\bullet}$ of the form (\ref{monad}), then
$E^\vee\simeq \calh^0(M_{\bullet}^{\vee})$ and $S_E:=\calh^1(M_{\bullet}^{\vee})$.
\end{lemma}
\begin{proof}
Dualizing (\ref{one}) and breaking (\ref{dual-two}) into short exact sequences, we obtain the following three short exact sequences:
\begin{equation}\label{dual-one}
0 \to C^\vee \stackrel{\beta^\vee}{\rightarrow} B^\vee \to K^\vee \to 0 ,
\end{equation}
\begin{equation}\label{dual-two1}
0 \to E^\vee \to K^\vee \stackrel{\alpha^\vee}{\rightarrow} T \to 0 ,
\end{equation}
where $T=\im(\alpha^\vee)$, and
\begin{equation}\label{dual-two2}
0 \to T \to A^\vee \to S_E \to 0.
\end{equation}

On the other hand, a complex of the form (\ref{dual-monad}) whose first map is injective can be broken down into three short exact sequences as follows
\begin{equation}\label{dual1}
0 \to C^\vee \stackrel{\beta^\vee}{\rightarrow} B^\vee \to V \to 0 ,
\end{equation}
where $V:=\coker(\beta^\vee)$,
\begin{equation}\label{dual2}
0 \to \calh^0(M_{\bullet}^{\vee}) \to V \stackrel{\alpha^\vee}{\rightarrow} T \to 0 ,
\end{equation}
where $T=\im(\alpha^\vee)$, and
\begin{equation}\label{dual3}
0 \to T \to A^\vee \to \calh^1(M_{\bullet}^{\vee}) \to 0.
\end{equation}

The desired conclusion follows from the comparison of the two sets of sequences.
\end{proof}

We have one more observation regarding reflexive sheaves on $3$-dimensional varieties.

\begin{lemma}\label{no.pt.sing}
Let $M_\bullet$ be a monad as in equation (\ref{monad}) whose degeneration locus has codimension at least $3$. Then its cohomology sheaf $E$ is reflexive.
\end{lemma}

\begin{proof}
Dualizing sequence (\ref{dual-two2}) we conclude that $T^\vee\simeq A$ and
$\inext^1(T,\op3)=\inext^2(S_E,\op3)=0$, because $S_E$ is supported in codimension at least $3$. Then dualizing sequence
(\ref{dual-two1}) we obtain, since $K$ is locally free: 
\begin{equation}\label{dual4}
0 \to A \stackrel{\alpha^{\vee\vee}}{\longrightarrow} K \simeq \to E^{\vee\vee} \to 0;
\end{equation}
in other words, $E^{\vee\vee}\simeq\coker\alpha^{\vee\vee}\simeq\coker\alpha\simeq E$, as desired.
\end{proof}


\section{Instantons on $\p3$}

We are finally in position to focus on our main object of study. In this Section, we review the definitions of instanton sheaves and rank $0$ instantons on $\p3$ and prove new basic results regarding their structure.


\subsection{Instanton sheaves}

Recall from \cite[p. 69]{J-i} that an \emph{instanton sheaf} on $\p3$ is a torsion free coherent sheaf $E$ with $c_1(E)=0$ and 
$$ h^0(E(-1))=h^1(E(-2))=h^2(E(-2))=h^3(E(-3))=0 . $$
Moreover, the integer $c:=h^1(E(-1))$ is called the \emph{charge} of $E$. One can check that it coincides with $c_2(E)$. 

As observed in the Introduction, locally free instanton sheaves of rank $2$ are precisely \emph{(mathematical) instanton bundles}. In fact, one can show that instanton sheaves are precisely those that can be obtained as the cohomology of a \emph{linear monad} of the form
\begin{equation}\label{instanton-monad}
\op3(-1)^{\oplus c} \stackrel{\alpha}{\rightarrow}
\op3^{\oplus r+2c} \stackrel{\beta}{\rightarrow} \op3(1)^{\oplus c},
\end{equation} 
where $r$ is the rank of $E$.

If $E$ is a locally free instanton sheaf on $\p3$, then it is easy to see that its dual $E^{\vee}$ and its double dual $E^{\vee\vee}$ are also instanton sheaves. The same is not true in general; in fact, we will see below that if $E$ is a reflexive instanton sheaf which is not locally free, then $E^\vee$ is not instanton. However, the sheaves $E^{\vee}$ and $E^{\vee\vee}$ still retains the following properties.

\begin{lemma}\label{dual-lemma2}
If $E$ is an instanton sheaf on $\p3$, then 
\begin{itemize}
\item[(i)] $h^0(E^{\vee}(-1))=h^1(E^{\vee}(-2))=0$;
\item[(ii)] $h^2(E^{\vee}(-2))=h^0(S_E(-2))$;
\item[(iii)] $h^2(E^{\vee\vee}(-2))=h^3(E^{\vee\vee}(-3))=0$;
\item[(iv)] $h^1(E^{\vee\vee}(-2))=h^1(Q_E(-2))$.
\end{itemize}
If, in addition, $E$ is either $\mu$-semistable or of trivial splitting type, then \linebreak $h^3(E^{\vee}(-3))=h^0(E^{\vee\vee}(-1))=0$.
\end{lemma}

\begin{proof}
Dualizing the monad (\ref{instanton-monad}) we obtain the complex
\begin{equation}\label{dual-instanton}
\op3(-1)^{\oplus c} \stackrel{\beta^\vee}{\rightarrow}
\op3^{\oplus r+2c} \stackrel{\alpha^\vee}{\rightarrow} \op3(1)^{\oplus c}.
\end{equation}
We know from Lemma \ref{dual-lemma1} that $E^\vee$ is the $0^{\rm th}$ cohomology of (\ref{dual-instanton}), while its first cohomology yields the sheaf $S_E$. Breaking (\ref{dual-instanton}) into short exact sequences as in the proof of Lemma \ref{dual-lemma1} 
and passing to cohomology, we obtain that \emph{(i)} and \emph{(ii)}.

Since $E$ is torsion free, we know that $\dim\sing(E)\le1$; since $\supp(Q_E)\subseteq\sing(E)$, we conclude that $\dim Q_E\le1$. Now we use the short exact sequence
\begin{equation}\label{dd-sqc}
0 \to E \to E^{\vee\vee} \to Q_E \to 0 
\end{equation}
to obtain \emph{(iii)} and \emph{(iv)}, since $h^2(Q_E(k))=h^3(Q_E(k))=0$ for every $k\in\ZZ$.

Finally, if $E$ is either $\mu$-semistable or of trivial splitting type, then the double dual sheaf $E^{\vee\vee}$ is, respectively, either $\mu$-semistable or of trivial splitting type; in either case, we have that $h^0(E^{\vee\vee}(-1))=0$. The vanishing of $h^3(E^{\vee}(-3))$ then follows from Lemma \ref{reflexive1}.

\end{proof}

The last result in this section describes the sheaves $\inext^p(S_E,\op3)$ when $E$ is an instanton sheaf. It also provides, in particular, a relation between the sheaves $S_E$ and $Q_E$.

\begin{lemma}\label{inext se}
If $E$ is an instanton sheaf on $\p3$, then 
\begin{itemize}
\item[(i)] $\inext^1(S_E,\op3)=0$;
\item[(ii)] $\inext^2(S_E,\op3) \simeq Q_E$;
\item[(iii)] $\inext^3(S_E,\op3)\simeq S_{E^\vee}$.
\end{itemize}\end{lemma}

\begin{proof}
Dualizing the sequence
\begin{equation}\label{ker-sqc2}
0 \to \op3(-1)^{\oplus c} \stackrel{\alpha}{\rightarrow} K \to E \to 0
\end{equation}
and breaking into short exact sequences we obtain
\begin{equation}\label{dual-ker1}
0 \to E^\vee \to K^\vee \stackrel{\alpha^\vee}{\rightarrow} T \to 0,
\end{equation}
where $T=\im(\alpha^\vee)$, and
\begin{equation}\label{dual-ker2}
0 \to T \to \op3(1)^{\oplus c} \to S_E \to 0,
\end{equation}
We can then gather (\ref{ker-sqc2}) and the dual of (\ref{dual-ker1}) into the following diagram
$$ \xymatrix{
& & & 0 \ar[d] & & \\
0 \ar[r] & \op3(-1)^{\oplus c} \ar[r]^{\alpha} & K \ar[r] \ar[d]^{\simeq} & E \ar[r] \ar[d] & 0 & \\
0 \ar[r] & T^\vee \ar[r]^{\alpha^{\vee\vee}} & K \ar[r] & E^{\vee\vee} \ar[r] \ar[d] & S_T \ar[r] & 0 \\
& & & Q_E \ar[d] & & \\
& & & 0 & &
} $$
where we recall that $S_T=\inext^1(T,\op3)$.

It follows that $T^\vee\simeq \op3(-1)^{\oplus c}$ and $\inext^1(T,\op3)\simeq Q_E$. Note also that \linebreak $\inext^2(T,\op3)\simeq\inext^1(E^\vee,\op3)$ and $\inext^3(T,\op3)=0$.

Dualizing (\ref{dual-ker2}), it follows that $\inext^1(S_E,\op3)=0$, 
$\inext^2(S_E,\op3)\simeq \inext^1(T,\op3)$ and $\inext^3(S_E,\op3)\simeq \inext^2(T,\op3)$.
\end{proof}


\subsection{Rank $0$ instantons}\label{rk0instantons}

The notion of rank $0$ instanton on $\p3$ was introduced in \cite[Defn. 6.1]{hauzer}. A coherent sheaf $Z$ on $\p3$ is called a \emph{rank $0$ instanton} if it has pure dimension $1$ and $h^0(Z(-2))=h^1(Z(-2))=0$. The integer $d:=h^0(Z(-1))$ is called the \emph{degree} of $Z$. 

One can show, see \cite[Lemma 6.2]{hauzer}, that $Z$ is a rank $0$ instanton if and only if there is a complex of the form (a.k.a. a \emph{perverse instanton})
$$ Z_\bullet ~~:~~ \op3(-1)^{\oplus d} \stackrel{\sigma}{\longrightarrow} \op3^{\oplus 2d}
\stackrel{\tau}{\longrightarrow} \op3(1)^{\oplus d} $$
such that $\calh^{-1}(Z_\bullet)=\calh^{0}(Z_\bullet)=0$ and $\calh^{1}(Z_\bullet)=Z$; here, $d$ is precisely the degree of $Z$.

In other words, $Z$ is a rank $0$ instanton of degree $d$ if and only if it admits a resolution of the form
\begin{equation}\label{z-res}
0 \to \op3(-1)^{\oplus d} \stackrel{\sigma}{\longrightarrow} \op3^{\oplus 2d} 
\stackrel{\tau}{\longrightarrow} \op3(1)^{\oplus d} \to Z \to 0.
\end{equation}
It follows immediately that $\inext^3(Z,\op3)=0$. Note also that $\inext^1(Z,\op3)=0$, since $\codim Z=2$.

\begin{lemma} \label{dual rk0}
If $Z$ is a rank $0$ instanton, then so is $\inext^2(Z,\op3)$.
\end{lemma}

\begin{proof}
Break (\ref{z-res}) into short exact sequences to obtain
\begin{equation}\label{z-res1}
0 \to \op3(-1)^{\oplus d} \stackrel{\sigma}{\longrightarrow} \op3^{\oplus 2d} \stackrel{\tau}{\longrightarrow} I \to 0,
\end{equation}
where $I=\im\tau=\coker\sigma$, and 
\begin{equation}\label{z-res2}
0 \to I \to \op3(1)^{\oplus d} \to Z \to 0.
\end{equation}

Dualizing (\ref{z-res2}), we conclude that $I^\vee\simeq \op3(-1)^{\oplus d}$, since $\inext^1(Z,\op3)=0$, and $\inext^1(I,\op3)\simeq \inext^2(Z,\op3)$. 

Thus dualizing (\ref{z-res1}) we obtain
\begin{equation}\label{z-res dual}
0 \to \op3(-1)^{\oplus d} \stackrel{\tau^\vee}{\longrightarrow} \op3^{\oplus 2d} 
\stackrel{\sigma^\vee}{\longrightarrow} \op3(1)^{\oplus d} \to \inext^2(Z,\op3) \to 0,
\end{equation}
thus $\inext^2(Z,\op3)$ is a rank $0$ instanton as well.
\end{proof}

As seen on the proof above, $\inext^2(Z,\op3)$ is obtained essentially be dualizing the resolution that defines $Z$. For this reason, we say that  $\inext^2(Z,\op3)$ is the \emph{dual} of $Z$. Two rank $0$ instantons $Z_1$ and $Z_2$ are dual to each other if $\inext^2(Z_1,\op3)\simeq Z_2$ and 
$\inext^2(Z_2,\op3)\simeq Z_1$.

Finally, we now analyze whether the sheaves $S_E$ and $Q_E$ are rank $0$ instantons. Note that
$\dim S_E\le1$ and $\dim Q_E\le1$, since $E$ is torsion free; however, both sheaves may have zero dimensional subsheaves. Our next two results provide sufficient conditions for $S_E$ and $Q_E$ to be rank $0$ instantons.

\begin{lemma}\label{se is rk0}
If $E$ is an instanton sheaf which is not locally free and such that $E^\vee$ is instanton, then $S_E$ is rank $0$ instanton.
\end{lemma}

\begin{proof}
If $E^\vee$ is instanton, then $h^0(S_E(-2))=h^0(E^\vee(-2))=0$, by Lemma \ref{dual-lemma2}, item (ii); in particular, $\supp(S_E)$ cannot have zero dimensional subsheaves, so it has pure dimension $1$. 

To see that $h^1(S_E(-2))=0$, note that the sequences (\ref{dual-ker1}) and (\ref{dual-ker2}) yields
$$ h^1(S_E(-2))=h^2(T(-2))=h^3(E^\vee(-2))=0, $$
since $E^\vee$ is instanton.
\end{proof}

\begin{lemma}\label{qe is rk0}
If $E$ is an instanton sheaf on $\p3$ which is not reflexive and such that $E^{\vee\vee}$ is instanton, then $Q_E$ is rank $0$ instanton. Moreover, there is a short exact sequence
\begin{equation}\label{ext2 qe}
0 \to S_{E^{\vee\vee}} \to S_E \to \inext^2(Q_E,\op3) \to 0 .
\end{equation} \end{lemma}
 
\begin{proof}
Consider the sequence
\begin{equation}\label{dd}
0 \to E \to E^{\vee\vee} \to Q_E \to 0 .
\end{equation}
If $E$ and $E^{\vee\vee}$ are both instantons, one immediately gets from the cohomology sequence that $h^0(Q_E(-2))=h^1(Q_E(-2))=0$, so $Q_E$ is a rank $0$ instanton. 

Dualizing (\ref{dd}) one obtains (\ref{ext2 qe}) using the fact that $\inext^1(Q_E,\op3)=0$ (because $Q_E$ is supported in codimension $2$) and $\inext^2(E^{\vee\vee},\op3)=0$ (because $E^{\vee\vee}$ is instanton).
\end{proof}


\subsection{Two examples}\label{sec-examples}

We complete this section with two examples that highlight the necessity of the hypothesis used in Lemmata
\ref{se is rk0} and \ref{qe is rk0}.

First, note that if $E$ is a reflexive instanton sheaf that is not locally free, then $\dim S_E=0$ and $S_E$ is not a rank $0$ instanton. Moreover, $E^\vee$ is not instanton.

Indeed, the first claim is clear, since the singular locus of relfexive sheaves must have codimension at least $3$. Since $\dim S_E=0$, then $h^0(S_E(-2))\ne0$, thus $E^\vee$ is not instanton by Lemma \ref{dual-lemma2}, item (ii).

Here is a concrete example of a reflexive instanton sheaf which is not locally free, taken from 
\cite[Example 5]{J-i}. Consider the following instanton monad,

\begin{equation}\label{example1}
\op3(-1) \stackrel{\alpha}{\rightarrow} \op3^{\oplus 5} \stackrel{\beta}{\rightarrow}\mathcal{O}_{\mathbb{P}^3}(1)
\end{equation}
where $\alpha$ and $\beta$ are defined by
$$ \alpha =\left(\begin{array}{r} -x_2 \\ x_1 \\ 0 \\ 0 \\ x_3 \end{array}\right)
~~ {\rm and} ~~ 
\beta = \left(\begin{array}{rrrrr} x_1 & x_2 & x_3 & x_4 & 0\end{array}\right). $$

Its cohomology sheaf is a rank $3$ instanton sheaf, here denoted $E$, of charge $1$. Its degeneracy locus, hence $\sing(E)$, consists of a single point, namely $[0:0:0:1]$, hence $E$ is reflexive, but not locally free. In particular, $E^\vee$ is not instanton.

We shall further discuss rank $3$ instantons of charge $1$ in Section \ref{sec r3 c1} below.

Next, we provide an example of an instanton sheaf which is not reflexive and such that $E^{\vee\vee}$ is not an instanton, taken from \cite[Example 3]{J-i}.

Consider first the following linear monad,
$$ 0 \to \op3(-1)^{\oplus 2} \to \op3^{\oplus 4} \to \op3(1) \to 0 ; $$
according to \cite[p. 5503]{F}, its cohomology sheaf is of the form ${\mathcal I}_C(1)$, the twisted ideal of a space curve $C\hookrightarrow\p3$. 

Floystad's result \cite[Main Theorem]{F} also guarantees, for any $c\ge1$, the existence of a
monad of the form:
$$ 0 \to \op3(-1)^{\oplus c} \to \op3^{\oplus 2c+4} \to \op3(1)^{\oplus c+1} \to 0 ~~, $$
whose cohomology is a rank $3$ locally free sheaf $F$ with $c_1(F)=-1$. 

The direct sum $E:=F\oplus {\mathcal I}_C(1)$ provides the desired example: $E$ is a non reflexive instanton sheaf of rank $4$ and charge $c+2$ ($c\ge 1$), being the cohomology of the linear monad
$$ 0 \to \op3(-1)^{\oplus c+2} \to \op3^{\oplus 2c+8} \to \op3(1)^{\oplus c+2} \to 0 ~~. $$

However, $E^{\vee\vee}=F^{\vee\vee}\oplus\op3(1)$ is not an instanton sheaf, since $H^0(E^{\vee\vee}(-1))\simeq H^0(\op3)\ne0$. Note that $Q_E\simeq {\mathcal O}_C(1)$, thus $Q_E$ may not be a rank $0$ instanton. Furthermore, note also that $E^\vee$ is locally free, but not an instanton either.


\section{Singular locus of rank $2$ instanton sheaves}

We are finally in position to establish the main result of this paper.



Let $E$ be a rank $2$ instanton sheaf on $\p3$. Applying Proposition \ref{reflexive} to its dual sheaf $F=E^\vee$, we conclude that $E^{\vee\vee}\simeq E^\vee$, since $\det(E^\vee)=\op3$. It then follows from Lemma \ref{dual-lemma2} that both $E^\vee$ and $E^{\vee\vee}$ are instanton sheaves. 

The fact that $S_E$ and $Q_E$ are rank $0$ instantons follow from Lemma \ref{se is rk0} and Lemma \ref{qe is rk0}, respectively. 

Furthermore, $E^\vee$ and $E^{\vee\vee}$ must be locally free by a direct application of Proposition \ref{prop-rog}, since $h^2(E^\vee(-2))=h^2(E^{\vee\vee}(-2))=0$. In particular, $S_{E^{\vee\vee}}=0$, thus $S_E\simeq \inext^2(Q_E,\op3)$ by sequence (\ref{ext2 qe}). Since also
$Q_E\simeq \inext^2(S_E,\op3)$ by Lemma \ref{inext se} item (ii), we have that $Q_E$ and $S_E$ are dual rank $0$ instantons.

Finally, since $\sing(E)=\supp(S_E)=\supp(Q_E)$, it follows that $\sing(E)$ has pure dimension $1$.  

Note, in addition, that if $E$ has charge $c$ and $E^{\vee\vee}$ has charge $c'$, then $Q_E$ and $S_E$ are rank $0$ instantons of degree $d:=c-c'$. In fact, their Hilbert polynomials are given by $P_{S_E}(k) = P_{Q_E}(k) = dk + 2d$, so that $Q_E$ may be regarded as points in the quot scheme ${\rm Quot}^{dk+2d}(E^{\vee\vee})$. 

To complete the proof of the Main Theorem, we assume that $d=1$. Since $Q_E$ is a rank $0$ instanton of degree $1$, $Q_E(-1)$ admits a resolution of the form
$$ 0 \to \op3(-2) \to \op3(-1)^{\oplus 2} \to \op3 \to Q_E(-1) \to 0. $$
Therefore, in fact, $Q_E(-1)\simeq \iota_*{\mathcal O}_\ell$ for some line $\iota:\ell\hookrightarrow\p3$. In other words, the singular locus of $E$ consists of a single line.

In particular, we also obtain the following claim.

\begin{corollary}\label{prop r2 c1}
Every rank $2$ non locally free instanton sheaf $E$ of charge $1$ is of the form
$$ 0 \to E \to \op3^{\oplus 2} \to \iota_*{\mathcal O}_\ell(1) \to 0, $$ 
for some line $\ell\in\p3$.
\end{corollary}

A complete classification of possible singular loci for rank $2$ instanton sheaves of charge $c$ seems to be a hard problem, since it requires an understanding of the quot schemes of rank $2$ locally free instanton sheaves. A procedure to construct instanton sheaves with a prescribed singular locus is given in \cite[Section 3]{JMT}. To be precise, let $E$ be an instanton sheaf of charge $c$, and consider triples $(\Sigma,L,\varphi)$ for $E$ consisting of the following:

\begin{itemize}
\item[(i)] an embedding $\iota:\Sigma\hookrightarrow\p3$ of a reduced, locally complete intersection curve of arithmetic genus $g$ and degree $d$;
\item[(ii)] an invertible sheaf $L\in{\rm Pic}^{g-1}(\Sigma)$ such that $h^0(\iota_*L)=h^1(\iota_*L)=0$;
\item[(iii)] a surjective morphism $\varphi:E\to\iota_*L(2)$.
\end{itemize}

It follows that $F:=\ker\varphi$ is an instanton sheaf of the same rank as $E$ and charge $c+d$; if $E$ is locally free, then $Q_F=\iota_*L(2)$, so that the singular locus of $F$ is precisely $\Sigma$. The difficulty, of course, is proving the existence of the surjective morphism $\varphi$. In \cite{JMT}, the cases of rational curves and elliptic quartic curves is considered.




\section{Singular locus of rank $3$ instanton sheaves} \label{rank3}

In this section we show that instanton sheaf of rank larger than $2$ may have $0$-dimensional singularities, as well as singular loci which are not of pure dimension. These phenomena first occur for instanton sheaves of rank $3$ and charge $1$ and $2$, respectively.


\subsection{Rank 3 instantons of charge 1} \label{sec r3 c1}

We will now consider linear monads of the form 
\begin{equation} \label{r=3 c=1}
\op3(-1) \stackrel{\alpha}{\rightarrow} \op3^{\oplus 5} \stackrel{\beta}{\rightarrow} \op3(1).
\end{equation}

Note that any surjective map $\op3^{\oplus 5} \to \op3(1)$ may, after a linear change of homogeneous coordinates and a change of basis on the free sheaf $\op3^{\oplus 5}$, be written in the following form:
$$ \beta = \left( \begin{array}{ccccc} x_1 & x_2 & x_3 & x_4 & 0 \end{array} \right). $$
It follows that the map $\alpha$ is given by
\begin{equation}\label{map-alpha}
\alpha = \left( \begin{array}{c} \sigma_1 \\ \sigma_2 \\ \sigma_3 \\ \sigma_4 \\ \phi \end{array} \right), \end{equation}
where $\sigma_j\in H^0(\op3(1))$ must satisfy the monad equation
\begin{equation} \label{monad-eqn}
\Sigma_j x_j\sigma_j = 0 .
\end{equation}
Moreover, the injectivity of $\alpha$ is equivalent to at least one of the sections $\sigma_j,\phi$ being non-trivial.

However, $\phi\equiv0$ if and only if the monad (\ref{r=3 c=1}) decomposes as a sum of two monads
$$ \left( \op3(-1) \to \op3^{\oplus 4} \to \op3(1) \right) \bigoplus 
\left( 0 \to \op3 \to 0 \right) $$
which in turns is equivalent to the cohomology of the (\ref{r=3 c=1}) splitting as a direct sum
$E\oplus\op3$, for some instanton sheaf $E$ of rank $2$ and charge $1$.

We assume therefore, from now on, that $\phi\ne0$. Let $\Gamma$ be the subspace of $H^0(\op3(1))$ spanned by the sections $\sigma_j$ and $\phi$; in particular, $\dim\Gamma\ge1$. 

We observe that $\dim\Gamma=1$ if and only if the cohomology of the monad (\ref{r=3 c=1}) splits as a sum $\Omega^1_{\p3}(1)\oplus{\mathcal O}_{\wp}$, where $\wp$ is the hyperplane defined by the equation $\{\phi=0\}$. Indeed, $\dim\Gamma=1$ if and only if $\sigma_j=\lambda_j\phi$ for each $j=1,2,3,4$; thus one can find a basis for the free sheaf $\op3^{\oplus 5}$ in which the map $\alpha$ is of the form
$$ \alpha = \left( \begin{array}{c} 0 \\ 0 \\ 0 \\ 0 \\ \phi \end{array} \right). $$
It follows that the monad (\ref{r=3 c=1}) decomposes as a sum of two linear monads
$$ \left( 0 \to \op3^{\oplus 4} \stackrel{\beta}{\rightarrow} \op3(1) \right) \bigoplus 
\left( \op3(-1) \stackrel{\phi}{\rightarrow} \op3 \to 0 \right) $$
which in turns is equivalent to the cohomology of the (\ref{r=3 c=1}) splitting as a direct sum
$\Omega^1_{\p3}(1) \oplus {\mathcal O}_{\wp}$.

\begin{proposition}
There exists a 1-1 correspondence between indecomposable instanton sheaves of rank $3$ and charge $1$ on $\p3$, and non-trivial extensions of ${\mathcal O}_{\wp}$ by $\Omega^1_{\p3}(1)$, for some hyperplane $\wp\subset\p3$.
\end{proposition}

\begin{proof}
As seen above, every indecomposable rank $3$ instanton sheaf $E$ of charge $1$ on $\p3$ is the cohomology of a monad of the form (\ref{r=3 c=1}) with $\phi\ne0$ and $\dim\Gamma\ge2$.
The key observation here is that, in this case, the monad (\ref{r=3 c=1}) can be written as a (non-trivial) extension of two simpler linear monads
\begin{equation}\label{ext monads} \xymatrix{
& 0 \ar[d] & \\
& \op3^{\oplus 4} \ar[d] \ar[r]^{\omega} & \op3(1) \ar[d]^{\simeq} \\
\op3(-1) \ar[d]^{\simeq} \ar[r]^{\alpha} & \op3^{\oplus 5} \ar[d] \ar[r]^{\beta} & \op3(1) \\
\op3(-1) \ar[r]^{\cdot\phi}\ar[d] & \op3 \ar[d]& \\
0 & 0 & } \end{equation}
where the map $\omega$ is given by
$$ \omega = \left( \begin{array}{ccccc} x_1 & x_2 & x_3 & x_4 \end{array} \right), $$
and $\alpha$ is given by (\ref{map-alpha}).

Passing to cohomology, this short exact sequence of complexes yields precisely the short exact sequence
\begin{equation} \label{r=3 c=1 ext}
0 \to \Omega^1_{\p3}(1) \to E  \to {\mathcal O}_{\wp} \to 0
\end{equation}
 between their middle cohomologies, where $\wp$  is the hyperplane defined by the equation $\phi=0$.

Conversely, if $E$ is a non-trivial extension of ${\mathcal O}_{\wp}$ by $\Omega^1_{\p3}(1)$, for some hyperplane $\wp\subset\p3$, then one can lift the short exact sequence
(\ref{r=3 c=1 ext}) to a short exact sequence of complexes as in (\ref{ext monads}). From the considerations above, we know that the linear monad thus obtained in the middle row is such that $\phi\ne0$ and $\dim\Gamma\ge 2$, thus its cohomology sheaf is an indecomposable rank $3$ instanton sheaf $E$ of charge $1$.
\end{proof}

The previous Proposition allows us to provide a neat description of the moduli space of indecomposable rank $3$ instanton sheaves of charge $1$, which we will denote here by $\cali^{\rm tf}(3,1)$. 

If $E$ is such an object, let $\wp_E$ be the corresponding hyperplane in $\p3$, obtained via the sequence (\ref{r=3 c=1 ext}); this yields a map 
$$ \varpi : \cali^{\rm tf}(3,1) \to (\p3)^\vee. $$

Next, note that $\varpi$ is surjective. Indeed, consider first the sequence
$$ 0 \to \op3(-1) \stackrel{\cdot\phi}{\longrightarrow} \op3 \to {\mathcal O}_{\wp} \to 0 ,$$
i.e. $\wp$ is the hyperplane within $\p3$ given by the equation $\phi=0$.
Applying the functor $\Hom(\cdot,\Omega^1_{\p3}(1))$, we conclude that
$$ \ext^1({\mathcal O}_{\wp},\Omega^1_{\p3}(1)) \simeq H^0(\Omega^1_{\p3}(2)). $$
In particular, we have $\dim\ext^1({\mathcal O}_{\wp},\Omega^1_{\p3}(1))=6$, independently of the hyperplane $\wp$. Thus for every hyperplane $\wp\in(\p3)^\vee$ one has non-trivial extensions by $\Omega^1_{\p3}(1)$, and every such extension defines an element $[E]\in\cali^{\rm tf}(3,1)$ such that $\varpi([E])=\wp$.

Furthermore, we also concluded that the fibres of $\varpi$ are precisely the projectivization of $\ext^1({\mathcal O}_{\wp},\Omega^1_{\p3}(1))$. 

Summarizing the above discussion, we have the following result.

\begin{proposition}
The moduli space $\cali^{\rm tf}(3,1)$ of indecomposable rank $3$ instanton sheaves of charge $1$ on $\p3$ is a projective variety of dimension $8$, given by the total space of a $\p5$-bundle over $(\p3)^\vee$.
\end{proposition}

We are finally able to charaterize the singular loci of rank $3$ instanton sheaves of charge $1$.

\begin{proposition}
The singular locus of a rank $3$ instanton sheaf $E$ of charge $1$ which is not locally free is either a point, if $E$ is reflexive, or a line, if $E$ is not reflexive.

\end{proposition}

\begin{proof}
First, let $E$ be a reflexive instanton sheaf of rank $3$ and charge $1$. Then $\dim\sing(E)=0$, and we must show that $\sing(E)$ is a point.

Indeed, recall that $\sing(E)$ coincides with the degeneration locus of the monad (\ref{r=3 c=1}), which is given by the common zeros of the sections $\sigma_j$ and $\phi$; such set has dimension zero if and only if it consists of a single point.

Note, in addition, that $E$ is reflexive if and only if $\dim\Gamma=3$. We have already proved the only if part; conversely, if $\dim\Gamma=3$, then $\sing(E)$ consists of a single point, hence $E$ must be reflexive by Lemma \ref{no.pt.sing}.

Now if $E$ is indecomposable and not reflexive, then $\dim\Gamma=2$, which means that the degeneration locus of the monad (\ref{r=3 c=1}), and hence $\sing(E)$, is a line. 

If, on the other hand, $E$ is decomposable (i.e. if $\dim\Gamma=1$), then it decomposes as a sum $E'\oplus \op3$ with $E'$ being a non locally free rank $2$ instanton of charge $1$. It then follows from Proposition \ref{prop r2 c1} that $\sing(E')$, which of course coincides with $\sing(E)$, is a line.   

%
%
\end{proof}

\bigskip

\begin{remark}\rm
It is easy to see that every reflexive instanton sheaf of rank $3$ on $\p3$ is $\mu$-semistable. Indeed, every instanton sheaf $E$ on $\p3$ satisfies $H^0(E(-1))=H^0(E^\vee(-1))=0$, thus $\mu$-semistability follows from the criterion in \cite[Remark 1.2.6 b, page 167]{OSS}.

On the other hand, by \cite[Prop. 16]{J-i}, there are no $\mu$-stable instanton sheaves of rank $3$ and charge $1$. Furthermore, one can check from (\ref{r=3 c=1 ext}) that $h^0(E)=1$, and it follows that $E$ is not (Gieseker) semistable either.
\end{remark}


\subsection{Rank $3$ instanton sheaves of charge $2$} \label{r3c2}

In the last part of this paper, we present two interesting examples of rank $3$ instanton sheaves of charge $2$.

We begin by showing how to construct a rank $3$ instanton sheaf of charge $2$ whose singular locus is the disjoint union of a line and a point.

The starting point is an indecomposable reflexive instanton sheaf $F$ of rank $3$ and charge $1$ which is not locally free. For instance, take the one obtained as cohomology of the monad (\ref{example1}); its singular locus is just the point $P=[0:0:0:1]$. One easily checks that it has trivial splitting type, that is, its restriction to a generic line is trivial.

Let $\iota:\ell\hookrightarrow\p3$ be a line in $\p3$ that does not contain the point $[0:0:0:1]$ and to which the restriction of $F$ is trivial. As we have checked above, the  sheaf $\iota_*\ol(1)$ is a rank $0$ instanton of degree $1$. Moreover, since $F|_\ell\simeq \ol^{\oplus 2}$, there are surjective maps $\varphi:F\to\iota_*\ol(1)$; let $E:=\ker\varphi$. 

From the short exact sequence
\begin{equation}\label{et on f}
0 \to E \to F \stackrel{\varphi}{\longrightarrow} \iota_*\ol(1) \to 0
\end{equation}
one easily checks that $E$ is a rank $3$ instanton sheaf of charge $2$. Indeed, since $F$ is an instanton sheaf, it is easy to see that $E$ is torsion free and that $c_1(E)=c_3(E)=0$ and $c_2(E)=2$. One also checks immediately from the cohomology sequence that
$$ h^0(E(-1))=h^1(E(-2))=h^2(E(-2))=h^3(E(-3))=0. $$

\begin{remark}\rm
The construction of the two previous paragraphs is a particular case of an \emph{elementary transformation} of an instanton sheaf, as outlined in \cite[Section 3]{JMT}.
\end{remark}

Note also that $E^{\vee}\simeq F^\vee$, thus in particular $E^{\vee\vee}\simeq F$, so $E^{\vee\vee}$ is also an instanton sheaf, and the sequence $0\to E\to E^{\vee\vee}\to Q_E\to 0$ matches the sequence (\ref{et on f}), thus $Q_E=\iota_*\mathcal{O}_{\ell}(1)$. In addition, sequence (\ref{ext2 qe}) reduces to, in this case,
$$ 0 \to \jmath_*{\mathcal O}_P \to S_E \to \iota_*\ol(1) , $$
where $\jmath$ denotes the inclusion of the point $P=[0:0:0:1]$ within $\p3$. It follows immediately that $\sing(E)=P\cup\ell$, which is not of pure dimension (neither $0$ or $1$).  

\bigskip

We conclude with an explicit example of a monad whose cohomology is a reflexive instanton sheaf of rank $3$ and charge $2$ whose singular locus consists of two distinct points. We take

\begin{equation} \label{r=3 c=2}
\op3(-1)^{\oplus 2} \stackrel{\alpha_{mn}}{\rightarrow} \op3^{\oplus 7} \stackrel{\beta_{mn}}{\rightarrow}
\op3(1) ^{\oplus 2}
\end{equation}
with
$$ \alpha_{mn} = \left( \begin{array}{cc}
x_2 & -x_4 \\ - m^2 x_4 & x_2 \\ -x_1 & x_3 \\ n^2 x_3 & -x_1 \\ -x_3 & \frac{1}{n}x_3 \\ 
-\frac{1}{m}x_3 - m x_4 & -\frac{1}{mn}x_3 - x_4 \\ x_3 - \frac{m}{n}x_4 & \frac{1}{m}(x_3+x_4)
\end{array} \right) $$
and 
$$ \beta_{mn} = \left( \begin{array}{ccccccc}
-x_1 & x_3 & -x_2 & x_4 & x_3 + x_4 & 0 & x_3 \\
n^2x_3 & -x_1 & m^2x_4 & -x_2 & x_4 & x_3 & \frac{1}{m}x_3
\end{array} \right) $$
where $n$ is a root of the equation $n^3+2n^2+n+1=0$ and $m=n+1/n$ (this guarantees that $\beta\alpha=0$). One checks that $\beta$ is surjective everywhere (since $m\ne\pm1$) and that $\alpha$ fails to be injective only at the points $[n:0:1:0]$ and $[0:m:0:1]$.


\end{document}